\documentclass{article}

\usepackage{cite}
\usepackage{amssymb}
\usepackage{amsmath}
\usepackage{amsfonts}
\usepackage{amsthm}
\usepackage{fancyhdr}
\usepackage[explicit]{titlesec}
\usepackage{titletoc}
\usepackage{booktabs}
\usepackage[all]{xy}
\usepackage[colorlinks=true,linkcolor=blue]{hyperref}

\usepackage{cancel}

\usepackage[mathscr]{euscript}

\usepackage{url}
\DeclareMathOperator\sep{Sep}
\DeclareMathOperator\spn{Span}

\usepackage{hyperref}

 \DeclareMathOperator{\actson}{\hspace{-1pt}\curvearrowright}
\newcommand{\htop}{h^{\mathrm{tp}}_{\mathrm{nv}}}
\newcommand{\hnv}{h_{\mathrm{nv}}}
\newcommand{\rest}{\upharpoonright}

\newtheorem{proposition}{Proposition}[section]

\newtheorem*{theorem*}{Theorem}
\newtheorem{theorem}{Theorem}[section]
\newtheorem{corollary}{Corollary}[section]
\newtheorem{definition}{Definition}[section]
\newtheorem{example}{Example}[section]
\newtheorem{lemma}{Lemma}[section]

\numberwithin{equation}{section}

\everymath{\displaystyle}

\begin{document}

\title{Naive entropy of dynamical systems}

\author{Peter Burton \footnote{Research partially supported by NSF grant DMS-0968710}}

\date{\today}

\maketitle

\begin{abstract} We study an invariant of dynamical systems called naive entropy, which is defined for both measurable and topological actions of any countable group. We focus on nonamenable groups, in which case the invariant is two-valued, with every system having naive entropy either zero or infinity. Bowen has conjectured that when the acting group is sofic, zero naive entropy implies sofic entropy at most zero for both types of systems. We prove the topological version of this conjecture by showing that for every action of a sofic group by homeomorphisms of a compact metric space, zero naive entropy implies sofic entropy at most zero. This result and the simple definition of naive entropy allow us to show that the generic action of a free group on the Cantor set has sofic entropy at most zero. We observe that a distal $\Gamma$-system has zero naive entropy in both senses, if $\Gamma$ has an element of infinite order. We also show that the naive entropy of a topological system is greater than or equal to the naive measure entropy of the same system with respect to any invariant measure.  \end{abstract}

\section{Introduction.}

A fundamental aspect of the theory of dynamical systems is the invariant known as entropy. Defined for both measurable and topological systems, this is a nonnegative real number which quantifies how random the given dynamics are. Entropy was introduced for measurable $\mathbb{Z}$-systems by Kolmogorov in \cite{Ko59} and Sinai in \cite{Sa59} and for topological $\mathbb{Z}$-systems by Adler, Konheim and McAndrew in \cite{AKM65}. In \cite{OrWe87}, Ornstein and Weiss extended much of entropy theory from $\mathbb{Z}$-systems to $\Gamma$-systems for amenable groups $\Gamma$. More recently, there has been significant progress in creating ideas of entropy for systems where the acting group is nonamenable. The most significant aspect of this new work is Bowen's theory of sofic entropy, initially developed by him for measurable systems in the papers \cite{Bow10}, \cite{Bow12a}, \cite{Bow14a} and \cite{Bow14}, and further developed for both types of systems by Kerr and Li in \cite{KeLi11a}, \cite{KeLi11},\cite{KeLi13a} and \cite{KeLi13} and by Kerr in \cite{Ke13} and \cite{Ke15}. Another thread is the concept of Rokhlin entropy, developed for measurable systems by Seward in \cite{Sew12}, \cite{Sew14a} and \cite{Sew14b}. In this paper we begin to study a third notion of entropy for general systems, called naive entropy. This idea was suggested by Bowen in \cite{Bow14} as the most direct way of generalizing the definition for $\mathbb{Z}$-systems. While he considered only the measurable context, a similar definition can be made for topological systems.\\
\\
Following an observation of Bowen, we show that if $\Gamma$ is a nonamenable countable group then any topological or measurable $\Gamma$-system has naive entropy either $0$ or $\infty$. Thus for nonamenable groups naive entropy is best understood as a dichotomy rather than an invariant. A natural question is to what extent the dichotomy between zero and infinite naive entropy corresponds to the dichotomy between zero and positive sofic entropy. Bowen has conjectured in \cite{Bow14} that zero naive entropy implies sofic entropy at most zero. In Section \ref{sec4} we prove the following topological version of this conjecture. Here $\htop$ is the naive topological entropy and $h^{\mathrm{tp}}_\Sigma$ is the sofic entropy with respect to a sofic approximation $\Sigma$.

\begin{theorem} \label{thm1} Let $\Gamma$ be a sofic group, let $\Gamma \actson X$ be a topological $\Gamma$-system and let $\Sigma$ be a sofic approximation to $\Gamma$. If $\htop(\Gamma \actson X) = 0$ then $h^{\mathrm{tp}}_{\Sigma}(\Gamma \actson X) \leq 0$. \end{theorem}

One advantage of naive entropy is that in many cases it is easy to see that a system has zero naive entropy. For example in Section \ref{sec24} we observe that if $\Gamma$ has an element of infinite order, then any distal $\Gamma$-system has zero naive entropy in both senses. This gives a partial answer to a question of Bowen. Furthermore, in Section \ref{sec25} we are able show that if $\Gamma$ is a free group, then the generic $\Gamma$-system with phase space the Cantor set has zero naive topological entropy. More precisely, if $X$ is a compact metric space and $\Gamma$ a countable group, let $\mathrm{A}_{\mathrm{top}}(\Gamma,X)$ denote the Polish space of topological $\Gamma$-systems with phase space $X$. We say a sequence $\left( \Gamma \actson^{a_n} X \right)_{n=1}^\infty \subseteq \mathrm{A}_{\mathrm{top}}(\Gamma,X)$ of $\Gamma$-systems converges to a system $\Gamma \actson^a X$ if for every $\gamma \in \Gamma$ the sequence of homeomorphisms corresponding to $\gamma$ in $a_n$ converges uniformly to the homeomorphism corresponding to $\gamma$ in $a$.

\begin{theorem}\label{thm2} Let $2^\mathbb{N}$ denote the Cantor set and let $\mathbb{F}$ be any countable free group. The set of topological $\mathbb{F}$-systems with zero naive entropy is comeager in $\mathrm{A}_{\mathrm{top}}\left(\mathbb{F},2^\mathbb{N} \right)$. \end{theorem}

Combining Theorems \ref{thm1} and \ref{thm2} we have the following corollary.

\begin{corollary} If $\mathbb{F}$ is a countable free group, then the set of $\mathbb{F}$-systems with sofic entropy at most $0$ is comeager in $\mathrm{A}_{\mathrm{top}}\left(\mathbb{F},2^\mathbb{N} \right)$.  \end{corollary}

Another natural question to ask is whether there is a relation between naive measure entropy and naive topological entropy. In Section \ref{sec23} we show half of such a variational principle. Let $\hnv$ denote the naive measure entropy.

\begin{theorem} \label{thm3} If $\Gamma \actson X$ is a topological $\Gamma$-system and $\mu$ is an invariant measure for $\Gamma \actson X$ then \[ \hnv(\Gamma \actson (X,\mu)) \leq \htop(\Gamma \actson X). \]  \end{theorem}

\subsection{Notational preliminaries.}

Throughout the paper $\Gamma$ will denote a countable discrete group. A measurable $\Gamma$-system $\Gamma \actson^a (X,\mu)$ consists of a standard probability space $(X,\mu)$ and measure-preserving action on $\Gamma$ on $(X,\mu)$, equivalently a homomorphism $a: \Gamma \to \mathrm{Aut}(X,\mu)$ where $\mathrm{Aut}(X,\mu)$ is the group of measure-preserving bijections from $(X,\mu)$ to itself. We use Kechris's convention from \cite{K} and write $\gamma^a$ instead of $a(\gamma)$ for $\gamma \in \Gamma$. We identify two measure-preserving bijections if they agree almost everywhere, and thus identify two $\Gamma$-systems $\Gamma \actson^a (X,\mu)$ and $\Gamma \actson^b (X,\mu)$ if $\gamma^a = \gamma^b$ almost everywhere for each $\gamma \in \Gamma$. \\
\\
A topological $\Gamma$-system $\Gamma \actson^a X$ consists of a compact metrizable space $X$ and a homomorphism $a:\Gamma \to \mathrm{Homeo}(X)$, where $\mathrm{Homeo}(X)$ is the group of homeomorphisms of $X$. As in the measurable case, we write $\gamma^a$ instead of $a(\gamma)$. If $\Gamma = \mathbb{Z}$ we use the standard notation and write $a(1) = T$, denoting the system by $(X,T)$ or $(X,\mu,T)$.\\
\\
For $n \in \mathbb{N}$, we let $[n]$ denote the set $\{1,\ldots,n\}$.

\subsection{Acknowledgments.}

We thank Alexander Kechris for introducing us to this topic, and Lewis Bowen for allowing us to read his preprint \cite{Bow14}. We also thank the anonymous referee for numerous helpful comments.

\subsection{Additional note.}

After communicating our results to Brandon Seward, he informed us that the measurable case of Bowen's naive entropy conjecture has been proved independently by a number of researchers including Miklos Abert, Tim Austin, Seward himself and Benjamin Weiss. This together with our Theorem \ref{thm3}, the variational principle for sofic entropy and the fact that a topological system with no invariant measure has sofic entropy $-\infty$ give an alternate, indirect proof of our Theorem \ref{thm1}. Our work was done independently of the (as yet unpublished) work of these authors on the measurable case.

\section{Naive entropy.} 

\subsection{Naive measure entropy.}

In this section we introduce the naive measure entropy of a dynamical system. Fix a measurable $\Gamma$-system $\Gamma \actson^a (X,\mu)$. All partitions considered will be assumed to be measurable. If $\alpha = (A_1,\ldots,A_n)$ is a finite partition of $(X,\mu)$ the Shannon entropy $H_\mu(\alpha)$ of $\alpha$ is defined by \[ H_\mu(\alpha) = - \sum_{i=1}^n \mu(A_i) \log(\mu(A_i)).\] If $\alpha$ and $\beta$ are partitions of $(X,\mu)$, the join $\alpha \vee \beta$ is the partition consisting of all intersections $A \cap B$ where $A \in \alpha$ and $B \in \beta$. We make a similar definition for the join $\bigvee_{i=1}^n \alpha_i$ of a finite family $(\alpha_i)_{i=1}^n$ of partitions. If $\alpha$ is partition and $\gamma \in \Gamma$ we let $\gamma^a \alpha$ be the partition $\{\gamma^a A: A \in \alpha \}$. For a finite set $F \subseteq \Gamma$ let $\alpha^F$ denote the partition $\bigvee_{\gamma \in F} \gamma^a \alpha$. If $(X,\mu,T)$ is a $\mathbb{Z}$-system and $F = [0,n]$ we write $\alpha_0^n$ for $\alpha^F$. Recall the classical definition of entropy for $\mathbb{Z}$-systems.

\begin{definition} \label{def1.1} Let $(X,\mu,T)$ be a measurable $\mathbb{Z}$-system. The dynamical entropy $h_\mu(\alpha)$ of a finite partition $\alpha$ is defined by \[ h_\mu(\alpha) = \inf_{n \in \mathbb{N}} \frac{1}{n} H_\mu\left(\alpha^n_0\right).\] The \textbf{measure entropy} $h(X,\mu,T)$ of the system is defined by \[ h(X,\mu,T) = \sup \{ h_\mu(\alpha): \alpha \mbox{ is a finite partition of } X. \} \] \end{definition}

See Chapter 14 of \cite{Glas03} for more information on the entropy of $\mathbb{Z}$-systems. In \cite{Bow14}, L. Bowen has introduced the following analog of Definition \ref{def10}.

\begin{definition} \label{def10} Let $\Gamma \actson (X,\mu)$ be a measurable $\Gamma$-system. The dynamical entropy $h_\mu(\alpha)$ of a finite partition $\alpha$ is defined by \[ h_\mu(\alpha) = \inf_F \frac{1}{|F|} H_\mu\left(\alpha^F \right),\] where the infimum is over all nonempty finite subsets $F$ of $\Gamma$. The \textbf{naive measure entropy} $\hnv\left(\Gamma \actson (X,\mu) \right)$ of the system is defined by \[ \hnv \left(\Gamma \actson (X,\mu)\right) = \sup \{ h_\mu(\alpha): \alpha \mbox{ is a finite partition of }X \}.\] \end{definition}

In the case of $\mathbb{Z}$, Theorem $4.2$ in \cite{DFR15} asserts that Definition \ref{def10} agrees with Definition \ref{def1.1}. The next fact was proven by Bowen in \cite{Bow14}.

\begin{theorem}\label{prop47} If $\Gamma$ is nonamenable then for any measurable $\Gamma$-system $\Gamma \actson (X,\mu)$ we have $\hnv(\Gamma \actson (X,\mu)) \in \{0,\infty\}$. \end{theorem}

\begin{proof} Suppose there is a finite partition $\alpha$ with $h_\mu(\alpha) = c > 0$. Choose $r \in \mathbb{R}$. Since $\Gamma$ is nonamenable, there is a finite set $W \subseteq \Gamma$ such that \[ \inf_F \frac{|WF|}{|F|} \geq \frac{r}{c}, \] where the infimum is over all nonempty finite subsets of $\Gamma$. Then we have \begin{align*} h_\mu\left(\alpha^W \right) &= \inf_F \frac{1}{|F|} H_\mu\left(\alpha^{WF} \right) \\ & = \inf_F \frac{|WF|}{|F|} \left( \frac{1}{|WF|} H_\mu \left(\alpha^{WF} \right)\right) \\& \geq \inf_F \frac{|WF|}{|F|} h_\mu(\alpha) \\ & \geq r. \end{align*} \end{proof}

\subsection{Naive topological entropy.}

In this section we introduce the naive topological entropy of a dynamical system. Fix a topological $\Gamma$-system $\Gamma \actson^a X$. If $\mathscr{U}$ is an open cover of a compact metric space $X$, let $N(\mathscr{U})$ denote the minimal cardinality of a subcover of $\mathscr{U}$. If $\mathscr{U}$ and $\mathscr{V}$ are open covers of $X$, the join $\mathscr{U} \vee \mathscr{V}$ is the open cover consisting of all intersections $U \cap V$ where $U \in \mathscr{U}$ and $V \in \mathscr{V}$. We make a similar definition for the join $\bigvee_{i=1}^n \mathscr{U}_i$ of a finite family $(\mathscr{U}_i)_{i=1}^n$ of open covers. If $\mathscr{U}$ is an open cover and $\gamma \in \Gamma$ we let $\gamma^a \mathscr{U}$ be the open cover $\{\gamma^a U: U \in \mathscr{U} \}$. For a finite set $F \subseteq \Gamma$, write $\mathscr{U}^F$ to refer to $\bigvee_{\gamma \in F} \gamma^a \mathscr{U}$. If $(X,T)$ is a $\mathbb{Z}$-system and $F = [0,n]$ we write $\mathscr{U}_0^n$ for $\mathscr{U}^F$. Again we recall the definition of entropy for $\mathbb{Z}$-systems.

\begin{definition} Let $(X,T)$ be a topological $\mathbb{Z}$-system. The entropy $h^{\mathrm{tp}}(\mathscr{U})$ of a finite open cover $\mathscr{U}$ is defined by \[ h^{\mathrm{tp}}(\mathscr{U}) = \inf_{n \in \mathbb{N}} \frac{1}{n} \log\left(N\left(\mathscr{U}_0^n \right)\right), \] and the \textbf{topological entropy} $h^{\mathrm{tp}}(X,T)$ of the system is defined by \[ h^{\mathrm{tp}}(\mathbb{Z} \actson X) = \sup \{h^{\mathrm{tp}}(\mathscr{U}): \mathscr{U} \mbox{ is a finite open cover of }  X \}.\] \end{definition} 

Following Definition \ref{def10} we make the following definition.

\begin{definition} Let $\Gamma \actson X$ be a topological $\Gamma$-system. Given a finite open cover $\mathscr{U}$ of $X$ we define the entropy $\htop(\mathscr{U})$ of $\mathscr{U}$ by \[ \htop(\mathscr{U}) =  \inf_F \frac{1}{|F|} \log \left( N \left(\mathscr{U}^F \right) \right), \] where the infimum is over all nonempty finite subsets of $\Gamma$. We define the \textbf{naive topological entropy} $\htop(\Gamma \actson X)$ of $\Gamma \actson X$ by    \[ \htop(\Gamma  \actson X) = \sup \{\htop(\mathscr{U}): \mathscr{U} \mbox{ is a finite open cover of }  X \}. \]  \end{definition}

A similar concept has been studied in \cite{Bis04}, \cite{BisUr06} and \cite{GLW88} and is discussed the text \cite{Down11}. If $\Gamma$ has a finite generating set $S$, these authors define the entropy of an open cover $\mathscr{U}$ by the formula \[ \limsup_{n \to \infty} \frac{1}{n} \log \left( N \left(\mathscr{U}^{S^n} \right) \right)\] and the entropy of the system by taking the supremum over finite open covers. Clearly a system with zero entropy in this sense has $\htop$ equal to zero. Hence we work with $\htop$ in order to get the strongest form of Theorem \ref{thm1}. An identical argument to the proof of Theorem \ref{prop47} shows that if $\Gamma$ is nonamenable then any topological $\Gamma$-system has naive topological entropy either $0$ or $\infty$.\\
\\
We record the following observation, which is immediate from the definition.

\begin{proposition} \label{prop25} If $\htop\left(\Gamma \actson^a X \right) > 0$ then for every $\gamma \in \Gamma$ with infinite order we have $h^{\mathrm{tp}}(X,\gamma^a) > 0$, where we regard $(X,\gamma^a)$ as a $\mathbb{Z}$-system. \end{proposition}

\subsection{Equivalent definitions of naive topological entropy.}

We now introduce two standard reformulations of the definition of naive topological entropy, due originally in the case of $\mathbb{Z}$ to R. Bowen. For a metric space $(X,d)$ and $\epsilon > 0$ say a set $S \subseteq X$ is $\epsilon$-separated if for each distinct pair $x_1,x_2 \in S$ we have $d(x_1,x_2) \geq \epsilon$. Say that $S$ is $\epsilon$-spanning if for every $x \in X$ there is $x_0 \in S$ with $d(x,x_0) \leq \epsilon$. Define $\sep(X,\epsilon,d)$ to be the maximal cardinality of an $\epsilon$-separated subset of $X$, and $\spn(X,\epsilon,d)$ to be the minimal cardinality of an $\epsilon$-spanning subset of $X$. It is clear that \begin{equation}\label{eq40}\spn(X,\epsilon,d) \leq \sep(X,\epsilon,d) \leq \spn\left(X,\frac{\epsilon}{2},d \right).\end{equation}

Now, fix a $\Gamma$-system $\Gamma \actson^a X$ and a compatible metric $d$ on $X$. For a nonempty finite subset $F \subseteq \Gamma$ define a metric $d_F$ on $X$ by letting $d_F(x_1,x_2) = \max_{\gamma \in F} d\left( \gamma^a x_1,\gamma^a x_2 \right)$. The proof of the following is an immediate generalization of the corresponding statement for $\mathbb{Z}$-systems, which can be found as Proposition $14.11$ in \cite{Glas03}.

\begin{proposition}\label{prop5} Letting $F$ range over the nonempty finite subsets of $\Gamma$ we have \[\htop\left(\Gamma \actson^a X \right) = \sup_{\epsilon > 0} \inf_F \frac{1}{|F|} \log( \sep(X,\epsilon,d_F)) = \sup_{\epsilon > 0} \inf_F \frac{1}{|F|} \log( \spn(X,\epsilon,d_F)).\] \end{proposition}

\begin{proof} Fix $\epsilon > 0$ and $F \subseteq \Gamma$ finite. Write $F^{-1}$ for $\{\gamma^{-1}:\gamma \in F\}$. Let $\mathscr{U}$ be an open cover of $X$ with Lebesgue number $\epsilon$. Let $S \subseteq X$ be an $\epsilon$-spanning set of minimal cardinality with respect to $d_{F^{-1}}$. For every $x \in X$ there is $s \in S$ with $d\left(\gamma^a x,\gamma^a s \right) \leq \epsilon$ for all $\gamma \in F^{-1}$. Write $B_\epsilon(s)$ for the ball of radius $\epsilon$ around $s$ with respect to $d$. We have $\gamma^a x \in B_\epsilon(\gamma^a s)$ or equivalently $x \in \left(\gamma^{-1} \right)^a B_\epsilon(\gamma^a s)$ for all $\gamma \in F^{-1}$. Therefore $x \in \bigcap_{\gamma \in F^{-1}} \left(\gamma^{-1} \right)^a B_\epsilon(\gamma^a s)$ and so $\bigcup_{s \in S}\bigcap_{\gamma \in F^{-1}} \left(\gamma^{-1} \right)^a B_\epsilon(\gamma^a s)$ is an open cover of $X$. Now, for every $s \in S$ and $\gamma \in F^{-1}$ we have that $B_\epsilon(\gamma^a s)$ is contained in some element of $\mathscr{U}$ and hence $\bigcap_{\gamma \in F^{-1}} \left(\gamma^{-1} \right)^a B_\epsilon(\gamma^a s)$ is contained in an element of $\mathscr{U}^F$. It follows that \begin{equation} \label{eq42} N\left(\mathscr{U}^F \right) \leq |S| = \spn \left(X,\epsilon,d_{F^{-1}} \right). \end{equation}

If $\mathscr{V}$ is an open cover of $X$, let $\mathrm{diam}(\mathscr{V})$ denote the supremum of the diameters of elements of $\mathscr{V}$. Let $\mathscr{V}$ be an open cover of $X$ with $ \mathrm{diam}(\mathscr{V}) \leq \epsilon$. Let $R$ be an $\epsilon$-separated set of maximal cardinality with respect to $d_F$. An element of $\mathscr{V}^F$ contains at most one point of $R$, and hence \begin{equation} \label{eq41} \sep\left(X,\epsilon,d_F \right) \leq N\left(\mathscr{V}^F \right). \end{equation}

By $(\ref{eq40}), (\ref{eq42})$ and $(\ref{eq41})$ if $\mathscr{U}$ has Lebesgue number $\epsilon$ and $\mathrm{diam}(\mathscr{V}) \leq \epsilon$ we have for all finite $F \subseteq \Gamma$: \begin{align} \htop(\mathscr{U}) &= \inf_F \frac{1}{|F|} \log \left(N \left(\mathscr{U}^F \right) \right) \nonumber \\ & \leq \inf_F \frac{1}{|F|} \log \left(\spn \left(X,\epsilon,d_F \right) \right) \nonumber \\ & \leq \inf_F \frac{1}{|F|} \log \left(\sep\left(X,\epsilon,d_F \right)\right) \nonumber \\& \leq \inf_F \frac{1}{|F|} \log \left(N\left(\mathscr{V}^F \right) \right) \nonumber \\ & = \htop(\mathscr{V}) \nonumber \\ & \label{eq43} \leq \htop\left(\Gamma \actson^a X \right). \end{align}

Assume $\htop\left(\Gamma \actson^a X \right) < \infty$. Given $\kappa > 0$ find an open cover $\mathscr{U}$ so that $\htop \left(\Gamma \actson^a X \right) - \kappa \leq \htop( \mathscr{U})$. Then if $\epsilon$ is less than the Lebesgue number of $\mathscr{U}$, $(\ref{eq43})$ implies that \begin{align*} \htop \left(\Gamma \actson^a X \right) - \kappa &\leq \inf_F \frac{1}{|F|} \log \left(\spn \left(X,\epsilon,d_F \right) \right) \\ & \leq \inf_F \frac{1}{|F|} \log \left(\sep\left(X,\epsilon,d_F \right)\right) \\ & \leq \htop\left(\Gamma \actson^a X \right).\end{align*}

Assume $\htop\left(\Gamma \actson^a X \right) = \infty$. Given $r \in \mathbb{R}$ find an open cover $\mathscr{U}$ so that $r \leq \htop( \mathscr{U})$. Then if $\epsilon$ is less than the Lebesgue number of $\mathscr{U}$, we have again by $(\ref{eq43})$ that \begin{align*} r \leq \inf_F \frac{1}{|F|} \log \left(\spn \left(X,\epsilon,d_F \right) \right) \leq \inf_F \frac{1}{|F|} \log \left(\sep\left(X,\epsilon,d_F \right)\right).\end{align*} \end{proof}

In particular we see from Proposition \ref{prop5} that the quantities \[\sup_{\epsilon > 0} \inf_F \frac{1}{|F|} \log( \sep(X,\epsilon,d_F))\] and \[\sup_{\epsilon > 0} \inf_F \frac{1}{|F|} \log( \spn(X,\epsilon,d_F)) \] are independent of the choice of compatible metric $d$.

\subsection{Proof of Theorem \ref{thm3}.} \label{sec23}

Recall that if $\alpha = (A_1,\ldots,A_k)$ and $\beta = (B_1,\ldots,B_m)$ are finite partitions of $(X,\mu)$, the conditional Shannon entropy $H(\alpha | \beta)$ of $\alpha$ given $\beta$ is defined by \[ H(\alpha|\beta) = - \sum_{i=1}^k \sum_{j=1}^m \mu(A_i \cap B_j) \log \left( \frac{ \mu(A_i \cap B_j)}{\mu(B_j)} \right). \]

We will use the following well-known facts about Shannon entropy, which appear in \cite{Glas03} as Propositions $14.16, 14.18.2$ and $14.18.4$ respectively.

\begin{proposition} \label{prop18} \begin{description} \item[(1)]  $H(\alpha_1 \vee \alpha_2) = H(\alpha_1) + H(\alpha_2|\alpha_1)$, in particular $H(\alpha_1 \vee \alpha_2) \geq H(\alpha_1)$,

\item[(2)] If $\beta_2$ refines $\beta_1$ then $H(\alpha|\beta_2) \leq H(\alpha|\beta_1)$,

\item[(3)] $H(\alpha_1 \vee \alpha_2 |\beta) \leq H(\alpha_1|\beta) + H(\alpha_2|\beta)$. \end{description} \end{proposition}

The following argument is a straightforward generalization of the corresponding proof for $\mathbb{Z}$-systems given as Part I of Theorem 17.1 in \cite{Glas03}.

\begin{proof}[Proof of Theorem \ref{thm3}]

Let $\mu$ be an invariant measure for the topological $\Gamma$-system $\Gamma \actson^a X$. Let $\alpha = (A_i)_{i=1}^k$ be a measurable partition of $(X,\mu)$. Choose closed sets $B_i \subseteq A_i$ such that $\mu(A_i \triangle B_i)$ is small enough so $H(\alpha|\beta) \leq 1$ where $\beta$ is the partition $(B_i)_{i=1}^{k+1}$ and $B_{k+1} = X - \bigcup_{i =1 }^k B_i$. Then for any finite set $F \subseteq \Gamma$ by $(2)$ and $(3)$ of Proposition \ref{prop18} we have \begin{align*} H_\mu \left( \alpha^F \big \vert \beta^F \right) & \leq \sum_{\gamma \in F} H_\mu \left(\gamma^a \alpha \big \vert \beta^F \right) \\ & \leq \sum_{\gamma \in F} H_\mu (\gamma^a \alpha \vert \gamma^a \beta) \\ &= |F| \cdot H_\mu(\alpha|\beta) \\ & \leq |F|. \end{align*}

Hence by $(1)$ of Proposition \ref{prop18} we have \begin{align*} H_\mu\left(\alpha^F \right) & \leq H_\mu \left( \alpha^F \vee \beta^F \right) \\ & = H_\mu \left( \beta^F \right) + H_\mu \left( \alpha^F \big \vert \beta^F \right) \\ &\leq H_\mu \left( \beta^F \right) + |F| \end{align*}

and consequently \begin{align} h_\mu(\alpha) &= \inf_{F} \frac{1}{|F|} H_\mu\left(\alpha^F \right) \nonumber \\&\leq \inf_{F} \frac{1}{|F|} \left( H_\mu \left(\beta^F \right) + |F| \right) \nonumber \\&= h_\mu(\beta) +1. \label{eq50} \end{align}

Now let $U_i = B_i \cup B_{k+1}$. Then $X - U_i = \bigcup_{\substack{1\leq j \leq k,\\ j \neq i}} B_j$ so $U_i$ is open and $\mathscr{U} = (U_i)_{i=1}^k$ is an open cover of $X$. Note that the only elements of $\beta$ meeting $U_i$ are $B_i$ and $B_{k+1}$. Let $\mathscr{V}(F)$ be an open subcover of $\mathscr{U}^F$ with minimal cardinality. We claim that each element of $\mathscr{V}(F)$ meets at most $2^{|F|}$ elements of $\beta^F$. Indeed suppose $\phi: F \to [k]$ is a function such that $\bigcap_{\gamma \in F} \gamma^a U_{\phi(\gamma)} \in \mathscr{V}(F)$ and let $x \in \bigcap_{\gamma \in F} \gamma^a U_{\phi(\gamma)}$. Then if $\psi:F \to [k+1]$ is any function so that $x \in \bigcap_{\gamma \in F} \gamma^a B_{\psi(\gamma)} \in  \beta^F$ we must have $B_{\psi(\gamma)} \cap U_{\phi(\gamma)} \neq \emptyset$ and hence $\psi(\gamma) \in \{\phi(\gamma),k+1\}$ for all $\gamma \in F$. Therefore \[\left \vert \beta^F \right \vert \leq 2^{|F|} \left \vert \mathscr{V}(F) \right \vert. \] It follows that \begin{align} H_\mu\left (\beta^F \right) & \leq \log \left( \left \vert \beta^F \right \vert \right) \nonumber \\ &\leq \log \left(2^{|F|} \cdot \left \vert \mathscr{V}(F) \right \vert \right ) \nonumber \\ & \leq |F| \log 2 + \log\left( \left \vert \mathscr{V}(F) \right \vert \right) \nonumber \\ & = |F| \log 2 + \log\left( N \left( \mathscr{U}^F \right) \right) \label{eq51} \end{align}

and hence by $(\ref{eq50})$ and $(\ref{eq51})$ we have \begin{align*} h_\mu(\alpha) & \leq h_\mu(\beta) +1 \\ & = \left( \inf_F \frac{1}{|F|} H_\mu \left( \beta^F \right) \right) +1 \\ & \leq \left( \inf_F \frac{1}{|F|} \left( |F| \log 2 + \log \Bigl( N \left( \mathscr{U}^F \right) \right) \Bigr) \right) +1 \\ & = \htop(\mathscr{U}) + 1 + \log 2. \end{align*}

Therefore \[ \hnv \left(\Gamma \actson (X,\mu) \right) \leq \htop \left( \Gamma \actson X \right) + 1 + \log 2. \] Now observe that the measure $\mu^n$ on $X^n$ is invariant for the $n^{\mathrm{th}}$ Cartesian power of the system $\Gamma \actson X$. Therefore the same argument shows \begin{equation} \label{eq52} h_{\mathrm{nv}}\left(\Gamma \actson \left(X^n,\mu^n\right) \right) \leq h_{\mathrm{nv}}^{\mathrm{top}} \left( \Gamma \actson X^n \right) + 1 + \log 2. \end{equation} 

Immediate generalizations of the proofs of Theorems $14.14$ and $14.31$ in \cite{Glas03} show that both forms of naive entropy are additive under direct products. Thus $(\ref{eq52})$ implies \[  n \cdot h_{\mathrm{nv}}\left(\Gamma \actson (X,\mu) \right) \leq n \cdot h_{\mathrm{nv}}^{\mathrm{top}} \left( \Gamma \actson X \right) + 1 + \log 2 \]for all $n \geq 1$ and therefore we must have \[h_{\mathrm{nv}}\left(\Gamma \actson (X,\mu) \right) \leq h_{\mathrm{nv}}^{\mathrm{top}} \left( \Gamma \actson X \right) .\] \end{proof}

\subsection{Examples.} \label{sec24}

\begin{example} Let $(Y,\nu)$ be a standard probability space. Assume $\nu$ is not supported on a single point. Consider the Bernoulli shift $\Gamma \actson \left(X,\mu \right)$ where $X = Y^\Gamma$ and $\mu = \nu^\Gamma$. Let $\alpha = (A_1,A_2)$ be a partition of $(Y,\nu)$ with positive entropy and $\hat{\alpha} = \left( \hat{A}_1,\hat{A}_2 \right)$ be the partition of $\left(X,\mu \right)$ given by \[\hat{A}_i = \left\{ \omega \in X: \omega\left(e_\Gamma \right) \in A_i \right \},\] where $e_\Gamma$ is the identity of $\Gamma$. Then as in the case of a $\mathbb{Z}$-system distinct shifts of $\hat{\alpha}$ are independent and so we have $H_\mu\left(\hat{\alpha}^F \right) = |F| \cdot H_\mu\left(\hat{\alpha}\right)$. Thus \[ h_\mu\left(\hat{\alpha} \right) = H_\mu\left(\hat{\alpha} \right) = H_\nu(\alpha) > 0. \] By Theorem \ref{prop47} we see that if $\Gamma$ is nonamenable then $\hnv\left(\Gamma \actson \left(X, \mu \right) \right) = \infty$. Thus Theorem \ref{thm3} implies that the corresponding topological system $\Gamma \actson X$ has infinite naive entropy. \end{example}

\begin{example} Let $\Gamma \actson^a X$ be a topological system and $d$ a compatible metric on $X$. Recall that $\Gamma \actson^a X$ is said to be distal if for every pair $x_1,x_2$ of distinct points in $X$ we have $\inf_{\gamma \in \Gamma} d\left(\gamma^a x_1,\gamma^a x_2 \right) > 0$. In particular, an isometric system such as a circle rotation is distal.\\
\\
Now, suppose that $\Gamma \actson^a X$ is distal and $\Gamma$ has an element $\gamma$ of infinite order. Then $(X,\gamma^a)$ is a distal $\mathbb{Z}$-system. Theorem $18.19$ in \cite{Glas03} implies that distal $\mathbb{Z}$-systems have zero entropy. Thus Proposition \ref{prop25} guarantees that $\htop(\Gamma \actson^a X) = 0$. By Theorem \ref{thm3}, $\hnv(\Gamma \actson^a (X,\mu) ) = 0$ for any invariant measure $\mu$. It is likely that a distal $\Gamma$-system has zero naive topological entropy for an arbitrary $\Gamma$, but we were unable to prove this despite significant effort. \end{example}

\subsection{Proof of Theorem \ref{thm2}} \label{sec25}

We first show three preliminary lemmas.

\begin{lemma}\label{lem20} Let $\mathscr{U}$ be a finite open cover of a compact metrizable space $X$. Fix a finite set $F \subseteq \Gamma$ and $k \in \mathbb{N}$. Then \[Z(\mathscr{U},F,k) = \Biggl \{ (\Gamma \actson^a X) \in \mathrm{A}_{\mathrm{top}}(\Gamma,X): N \Biggl( \bigvee_{\gamma \in F} \gamma^a \mathscr{U} \Biggr) \leq k \Biggr \} \] is open. \end{lemma}

\begin{proof} Write $\mathscr{U} = (U_i)_{i=1}^n$. Let $(\Gamma \actson^a X) \in Z(\mathscr{U},F,k)$ and let $\mathscr{V}$ be a subcover of $\bigvee_{\gamma \in F} \gamma^a \mathscr{U}$ with cardinality $\leq k$. Let $d$ be a compatible metric on $X$ and let $d_u$ be the metric \[d_u(f,g) = \sup_{x \in X} d(f(x),g(x)). \] Note that to obtain the uniform topology on $\mathrm{Homeo}(X)$ we must use the metric \[d'_u(f,g) = d_u(f,g) + d_n(f^{-1},g^{-1}).\] However the topology induced by $d_u$ on $\mathrm{A}_{\mathrm{top}}(\Gamma,X)$ is the same as the one induced by $d'_u$ so we will continue to work with the former.\\
\\
Let $\epsilon$ be a Lebesgue number for $\mathscr{V}$ with respect to $d$. Let $(\phi_j)_{j=1}^k$ be a sequence of functions from $F$ to $[n]$ so that \[ \mathscr{V} = \Biggl ( \bigcap_{ \gamma \in F} \gamma^a U_{\phi_j(\gamma)}  \Biggr )_{j=1}^k. \]  Let $\delta > 0$ be small enough that for all $\gamma \in F$ and $x_1,x_2 \in X$, $d(x_1,x_2) < \delta$ implies $d(\gamma^a x_1,\gamma^a x_2) < \epsilon$. Then for any $x \in X$, $\left(\gamma^{-1} \right)^a B_\epsilon(x)$ contains $B_\delta\left( \left(\gamma^{-1} \right)^a x \right)$. Suppose $d_u\left( \left(\gamma^{-1} \right)^a,\left(\gamma^{-1} \right)^b \right) < \delta$ for all $\gamma \in F$. We claim \[ \Biggl ( \bigcap_{ \gamma \in F} \gamma^b U_{\phi_j(\gamma)}  \Biggr )_{j =1}^k \] is a cover of $X$. Let $x \in X$. Then there is $j \leq k$ so that $B_\epsilon(x) \subseteq  \bigcap_{ \gamma \in F} \gamma^a U_{\phi_j(\gamma)}$, equivalently $\left( \gamma^{-1} \right)^a B_\epsilon(x) \subseteq U_{\phi_j(\gamma)}$ for all $\gamma \in F$. Since $d\left( \left(\gamma^{-1} \right)^a x, \left(\gamma^{-1} \right)^b x \right) < \delta$, we see that $\left( \gamma^{-1} \right)^b x \in U_{\phi_j(\gamma)}$. Therefore $x \in \gamma^b U_{\phi_j(\gamma)}$ for all $\gamma \in F$. \end{proof}

\begin{lemma}\label{lem21} For any system $\Gamma \actson X$, if $(\mathscr{U}_n)_{n=1}^\infty$ is a sequence of finite open covers such that $\lim_{n \to \infty} \mathrm{diam}(\mathscr{U}_n) = 0$, then $\lim_{n \to \infty} h^{\mathrm{tp}}(\mathscr{U}_n) = \htop(\Gamma \actson X)$. \end{lemma}

\begin{proof} It is clear that if $\mathscr{U}$ refines $\mathscr{V}$ then $h^{\mathrm{tp}}(\mathscr{V}) \leq h^{\mathrm{tp}}(\mathscr{U})$. Thus if $\mathscr{V}$ is an arbitrary open cover of $X$, by choosing $n$ so that $\mathrm{diam}(\mathscr{U}_n)$ is less than the Lebegsue number of $\mathscr{V}$ we have $h^{\mathrm{tp}}(\mathscr{V}) \leq h^{\mathrm{tp}}(\mathscr{U}_n)$. \end{proof}

\begin{lemma} \label{lem22} For any countable group $\Gamma$ and compact metrizable space $X$, the set of systems with zero naive topological entropy is $G_\delta$ in $\mathrm{A}_{\mathrm{top}}(\Gamma,X)$. \end{lemma}

\begin{proof} If $\mathscr{U}$ is an open cover of $X$, $F \subseteq \Gamma$ is finite and $\epsilon > 0$ set \[ \tilde{Z}(\mathscr{U},F,\epsilon) = \left\{ (\Gamma \actson^a X) \in \mathrm{A}_{\mathrm{top}}(\Gamma,X) : \frac{1}{|F|} \log\left( N \left( \bigvee_{\gamma \in F} \gamma^a \mathscr{U} \right) \right) < \epsilon \right \}.  \] Note that in the notation of Lemma \ref{lem20}, we have \[ \tilde{Z}(\mathscr{U},F,\epsilon)  = Z \left(\mathscr{U},F, \lfloor \exp(\epsilon|F|) \rfloor \right) \] hence $\tilde{Z}(\mathscr{U},F,\epsilon)$ is open. If $(\mathscr{U}_n)_{n=1}^\infty$ is a sequence of finite open covers with $\lim_{n \to \infty} \mathrm{diam}(\mathscr{U}_n) = 0$ then by Lemma \ref{lem21}, the set of systems with zero naive topological entropy is equal to the $G_\delta$ set \[ \bigcap_{n=1}^\infty \bigcap_{k=1}^\infty \bigcup_F \tilde{Z}\left(\mathscr{U}_n,\frac{1}{k},F \right), \] where the union is over all nonempty finite subsets of $\Gamma$. \end{proof}

\begin{proof}[Proof of Theorem \ref{thm2}.]  By Lemma \ref{lem22}, it suffices to show the set of systems with zero entropy is dense in $\mathrm{A}_{\mathrm{top}}\left(\Gamma, 2^{\mathbb{N}} \right)$. By Corollary 2.5 in \cite{GW01}, the set of homeomorphisms with zero entropy is uniformly dense in $\mathrm{Homeo} \left(2^\mathbb{N} \right)$. Therefore the set of systems in $\mathrm{A}_{\mathrm{top}}\left(\Gamma,2^\mathbb{N} \right)$ for which the first generator of $\Gamma$ acts with zero entropy is dense. The theorem follows from this fact and Proposition \ref{prop25}. \end{proof}

\section{Sofic groups and sofic entropy.}

\subsection{Sofic groups.}

Sofic groups were introduced by Gromov in \cite{Gro99} and Weiss in \cite{Wei00}. Let $\mathrm{Sym}(n)$ denote the symmetric group on $n$ letters. Let $u_n$ denote the uniform probability measure on $[n]$ so that $u_n(A) = \frac{|A|}{n}$. In keeping with our convention for dynamical systems, if $\sigma$ is a function from $\Gamma$ to $\mathrm{Sym}(n)$ we write $\gamma^\sigma m$ for $\sigma(\gamma)(m)$.

\begin{definition}\label{def1} Let $\Gamma$ be a countable discrete group. Let $\Sigma = (\sigma_i)_{i=1}^\infty$ be a sequence of functions $\sigma_i: \Gamma \to \mathrm{Sym}(n_i)$ such that $n_i \to \infty$. Note that the $\sigma_i$ are not assumed to be homomorphisms. We say $\Sigma$ is a \textbf{sofic approximation} to $\Gamma$ if for every pair $\gamma_1,\gamma_2 \in \Gamma$ we have \[ \lim_{i \to \infty} u_{n_i}(\{m \in [n_i]: (\gamma_1 \gamma_2)^{\sigma_i} m = \gamma_1^{\sigma_i} \gamma_2^{\sigma_i} m \}) = 1,\] and for every pair $\gamma_1 \neq \gamma_2$ we have \[ \lim_{i  \to \infty} u_{n_i} (\{m \in [n_i]: \gamma_1^{\sigma_i} m \neq \gamma_2^{\sigma_i} m \}) = 1.\] We say $\Gamma$ is \textbf{sofic} if there exists a sofic approximation to $\Gamma$. \end{definition}

Thus the first condition guarantees that the $\sigma_i$ are asymptotically homomorphisms, and the second condition guarantees that the corresponding approximate actions on $[n_i]$ are asymptotically free. The standard examples of sofic groups are residually finite groups and amenable groups. It is unknown whether every countable group is sofic.

\subsection{Topological sofic entropy.}

In \cite{KeLi11} and \cite{KeLi13}, Kerr and Li developed a topological counterpart to Bowen's theory of sofic entropy, based initially on operator-algebraic considerations. We will use the `spatial' formulation of these ideas. Fix a group $\Gamma$ and a topological $\Gamma$-system $\Gamma \actson^a X$. Fix a compatible metric $d$ for $X$. Define the metrics $d^2$ and $d^\infty$ on the set of maps from $[n]$ to $X$ by \[ d^2(\phi,\psi) = \left( \frac{1}{n} \sum_{m=1}^n d\left( \phi(m),\psi(m) \right)^2 \right)^{\frac{1}{2}} \] and \[ d^\infty(\phi,\psi) = \max_{m \in [n]} d(\phi(m),\psi(m)).\]

\begin{definition} Let $F \subseteq \Gamma$ be finite, $\delta > 0$ and $\sigma: \Gamma \to \mathrm{Sym}(n)$. Define $\mathrm{Map}(\sigma,F,\delta)$ to be the collection of functions $\phi: [n] \to X$ such that $d^2(\phi \circ \gamma^\sigma,\gamma^a \circ \phi) \leq \delta$ for all $\gamma \in F$. \end{definition}

\begin{definition} Let $\Sigma = (\sigma_i)_{i=1}^\infty$ be a sofic approximation to $\Gamma$ with $ \sigma_i \in \mathrm{Sym}(n_i)^\Gamma$. Define the \textbf{topological sofic entropy} $h^{\mathrm{tp}}_\Sigma(\Gamma \actson^a X)$ of $\Gamma \actson^a X$ with respect to $\Sigma$ as follows. Letting $F$ range over the nonempty finite subsets of $\Gamma$ and $\delta,\epsilon > 0$ define \begin{align*} h^{\mathrm{tp}}_\Sigma(\delta,F,\epsilon) &= \limsup_{i \to \infty} \frac{1}{n_i} \log( \sep(\mathrm{Map}(\sigma_i,F,\delta),\epsilon,d^\infty)), \\ h^{\mathrm{tp}}_\Sigma(F,\epsilon) &= \inf_{\delta > 0} h^{\mathrm{tp}}_\Sigma(\delta,F,\epsilon), \\ h^{\mathrm{tp}}_\Sigma(\epsilon) &= \inf_F h^{\mathrm{tp}}_\Sigma(F,\epsilon), \\ h^{\mathrm{tp}}_\Sigma(\Gamma \actson^a X) &= \sup_{\epsilon > 0} h^{\mathrm{tp}}_\Sigma(\epsilon).\end{align*} \end{definition}

\section{Proof of Theorem \ref{thm1}}\label{sec4}

This argument builds on the framework used to prove Lemma 5.1 in \cite{KeLi13}.

\subsection{Choosing parameters}

In this subsection we set the values of some initial parameters for our construction. Let $\Sigma = (\sigma_n)_{n=1}^\infty$ be a sofic approximation to $\Gamma$, where $\sigma_n: \Gamma \to \mathrm{Sym}(n)$. The case where $\sigma_n$ is a function from $\Gamma$ to $[k_n]$ for some $k_n \neq n$ can be handled with trivial modifications. Choose $\kappa$ with $0 < \kappa <1$. It suffices to show that $h^{\mathrm{tp}}_{\Sigma}(\Gamma \actson^a X) \leq \kappa$. Choose $\epsilon > 0$, so that it suffices to show that $h^{\mathrm{tp}}_{\Sigma}(\epsilon) \leq \kappa$. Let \begin{equation} \label{eq30} \eta = \frac{\kappa}{4 \log\left(\sep\left(X,\frac{\epsilon}{2},d\right)\right)} \end{equation} and choose $k \in \mathbb{N}$ such that \begin{equation}\label{eq31} \frac{1}{k} \leq \frac{\eta}{2}.\end{equation} By our assumption that $\htop(\Gamma \actson^a X) = 0$, we can choose a finite set $F \subseteq \Gamma$ such that \begin{equation} \label{eq1} \frac{1}{|F|} \log \left(\sep\left(X,\frac{\epsilon}{4},d_F \right) \right) \leq \frac{\kappa}{4k}. \end{equation} 

\begin{lemma}\label{lem4} Let $F' \subseteq F$ be such that $|F'| \geq \frac{|F|}{k}$. Then \[\sep\left(X,\frac{\epsilon}{4},d_{F'} \right) \leq \exp\left(\frac{\kappa |F'|}{4} \right).\] \end{lemma}

 \begin{proof}[Proof of Lemma \ref{lem4}] Since \[\sep\left(X,\frac{\epsilon}{4},d_{F'} \right) \leq \sep\left(X,\frac{\epsilon}{4},d_F \right),\] we have \begin{align*} \frac{1}{|F'|} \log \left( \sep\left(X,\frac{\epsilon}{4},d_{F'} \right) \right) & \leq \frac{1}{|F'|} \log\left( \sep\left(X,\frac{\epsilon}{4},d_F \right) \right) \\ & \leq k \left(\frac{1}{|F|} \log\left(\sep\left(X,\frac{\epsilon}{4},d_F \right)\right) \right) \\ & \leq \frac{\kappa}{4} \end{align*} where the last inequality follows from $(\ref{eq1})$. \end{proof}

Write $s = |F|$. Let $\delta > 0$ be small enough that  \begin{equation} \label{eq8} \delta \leq \left(\frac{\epsilon}{8} \right)^2,\end{equation} \begin{equation} \label{eq14} \delta \leq \frac{\eta}{4s^3 } \end{equation} (so in particular $s \delta < 1$) and finally \begin{equation} \label{eq7} -(s \delta \log(s \delta) + (1- s \delta) \log(1-s \delta)) \leq \frac{\kappa}{4}.\end{equation}

For a finite $S \subseteq \Gamma$ let \begin{align*} Q(S)_n & = \{m \in [n]: (\gamma_1 \gamma_2)^{\sigma_n} m = \gamma_1^{\sigma_n} \gamma_2^{\sigma_n} m \mbox{ for all } \gamma_1,\gamma_2 \in S \} \\ & \cap \{m \in [n]: \gamma_1^{\sigma_n} m \neq \gamma_2^{\sigma_n} m  \mbox{ for all } \gamma_1 \neq \gamma_2 \in S \}  \end{align*}

Write $\hat{F}$ for the symmetrization of $F$. Since $\Sigma$ is a sofic approximation, we can find $N$ so that if $n \geq N$ then \begin{equation}\label{eq60} \bigl \vert Q(\hat{F})_n \bigr \vert  \geq \left(1-  \frac{\eta}{4s^2} \right) n. \end{equation}

\subsection{Choosing a separated subset}

In this subsection we find a large $\epsilon$-separated subset $V$ of $\mathrm{Map}(\sigma,F,\delta)$ such that every element of $V$ is approximately equivariant on a fixed large subset of $[n]$. Fix $n \geq N$ and write $\sigma = \sigma_n$. Let $D$ be an $\epsilon$-separated subset of $\mathrm{Map}(\sigma,F,\delta)$ with respect to $d^\infty$ of maximal cardinality. For every $\phi \in \mathrm{Map}(\sigma,F,\delta)$ by definition we have $d^2(\phi \circ \gamma^\sigma, \gamma^a \circ \phi) \leq \delta$ for all $\gamma \in F$. Explicitly, \[ \left( \frac{1}{n} \sum_{m=1}^n d\bigl(\phi( \gamma^\sigma m),\gamma^a\phi(m)\bigr) ^2 \right)^{\frac{1}{2}} \leq \delta.\] Hence for each fixed $\gamma \in F$ at least $(1-\delta)n$ elements $m$ of $[n]$ have $d\left(\phi\left(\gamma^\sigma m\right),\gamma^a \phi(m)\right) \leq \sqrt{\delta}$. Hence the set $\Theta_\phi$ of all $m \in [n]$ such that $d\left(\phi \left(\gamma^\sigma m \right),\gamma^a \phi(m)\right) \leq \sqrt{\delta}$ for all $\gamma \in F$ has size at least $(1-s\delta) n$. \\
\\
By a standard estimate from information theory (see for example Lemma $16.19$ in \cite{FlGr06}) the number of subsets of $[n]$ of size at most $s \delta n$ is at most \[ \exp \bigl(-n(s \delta \log(s \delta) + (1- s \delta) \log(1-s \delta)) \bigr)\] and by (\ref{eq7}) this is bounded above by $ \exp \left( \frac{\kappa n}{4} \right)$. Hence there at at most $\exp \left( \frac{\kappa n}{4} \right)$ possible choices for the sets $\{\Theta_\phi: \phi \in  D \}$ and thus there are at least $\exp\left(-\frac{\kappa n}{4} \right)|D|$ elements of $D$ for which $\Theta_\phi$ is the same. So we can find $V \subseteq D$ and $\Theta \subseteq [n]$ such that \begin{equation} \label{eq21} |D| \leq \exp \left( \frac{\kappa n}{4} \right) |V| \end{equation} and for all $\phi \in V$ we have $\Theta_\phi = \Theta$. Note that since $|\Theta| \geq (1 - s \delta)n$, $(\ref{eq14})$ implies that \begin{equation} \label{eq10} |[n] - \Theta| \leq \frac{\eta n}{4s^2}. \end{equation} Furthermore, by $(\ref{eq8})$ and the definition of $\Theta$, for all $\phi \in V$ and all $m \in \Theta$ we have \begin{equation} \label{eq26} d\left(\phi \left(\gamma^\sigma m \right),\gamma^a \phi(m)\right) \leq \frac{\epsilon}{8}. \end{equation}

\subsection{Disjoint subsets of the sofic graph}

Endow $[n]$ with the structure of the graph $G_\sigma$ corresponding to $\sigma$, where $m_1$ is connected to $m_2$ if and only if there is $\gamma \in F$ such that $(\gamma)^\sigma m_1 = m_2$ or $\left(\gamma^{-1} \right)^\sigma m_1 = m_2$. In this section we find a maximal collection of disjoint subsets of $G_\sigma$ which resemble a nontrivial part of $F$.\\
\\
By $(\ref{eq60})$ and $(\ref{eq10})$, \[|G_\sigma - (Q(\hat{F})_n \cap \Theta)| \leq \frac{\eta n}{2s^2}.\]  Let $J$ be the collection of points $c$ in $G_\sigma$ such that the ball of radius $1$ around $c$ in $G_\sigma$ is contained in $Q(\hat{F})_n \cap \Theta$, and let $I$ be the collection of points $c$ in $J$ such that the ball of radius $1$ around $c$ is contained in $J$. Then \[ |G_\sigma - J| \leq s \cdot |G_\sigma - (Q(\hat{F})_n \cap \Theta)| \leq \frac{\eta n }{2s} \] and \begin{equation} \label{eq15} |G_\sigma - I| \leq s \cdot |G_\sigma - J| \leq \frac{\eta n}{2}. \end{equation}

If $c \in J$ then the mapping from $F$ to $G_\sigma$ given by $\gamma \mapsto \gamma^\sigma c$ is injective. We now begin an inductive procedure. Choose $c_1 \in J$ and take $F_1 = F$. Suppose we have chosen $c_1,\ldots,c_j \in J$ and $F_1,\ldots,F_j \subseteq F$ such that the sets $\left(F_i^\sigma c_i \right)_{i=1}^j$ are pairwise disjoint and $\frac{|F|}{k} \leq |F_i|$ for all $i \in \{1,\ldots,j\}$. Write $F_i^\sigma c_i = B_i$ \\
\\
Assume we cannot extend this process further, so that there do not exist $c_{j+1}$ and $F_{j+1}$ satisfying the two conditions. Write $W = \bigcup_{i=1}^j B_i$. Our assumption implies that for every $c \in J$, at least $\left(1-\frac{1}{k} \right) |F|$ of the points in $F^\sigma c$ lie in $W$. Suppose toward a contradiction that $\frac{|J|}{k} < |I - W|$. For each point $b$ in $I$, there are exactly $|F|$ points $c \in J$ such that $b \in F^\sigma c$, in symbols $|\{c \in J:b \in F^\sigma c\}| = |F|$. Indeed $b \in F^{\sigma}c$ if and only if $b = \gamma^\sigma c$ for some $c \in F$. Since $b,c \in Q(F)_n$, this is equivalent to $ \left(\gamma^{-1} \right)^\sigma b = c$. Since $b \in Q(F^{-1})_n$, the map $\gamma^{-1} \mapsto \left(\gamma^{-1}\right)^\sigma b$ is injective. Therefore \begin{align*} |\{c \in J:b \in F^\sigma c\}| & = |\{c \in J: c \in \left( F^{-1} \right)^\sigma b \}| \\ & = |F^{-1}| \\ &= |F|. \end{align*}

So we have \[ \sum_{b \in I - W} |\{c \in J:b \in F^\sigma c\}| = |F| \cdot |I - W| > \frac{|F| \cdot |J|}{k}. \] We can write \[ \sum_{b \in I - W} |\{c \in J:b \in F^\sigma c\}| = \sum_{b \in I - W} \sum_{c \in J} \mathbf{1}_{F^\sigma c}(b),\] where $\mathbf{1}_Y$ is the characteristic function of $Y$. So we have \[ \sum_{c \in J} \sum_{b \in I - W} \mathbf{1}_{F^\sigma c}(b) > \frac{|F| \cdot |J|}{k}. \] Since there are $|J|$ terms in the outer sum, there must be some $c_0 \in J$ with \[\sum_{b \in I - W} \mathbf{1}_{F^\sigma c_0}(b) > \frac{|F|}{k},\] or equivalently $|(I - W ) \cap F^\sigma c_0| > \frac{|F|}{k}$. Thus $|W \cap F^\sigma c_0| < \left(1 - \frac{1}{k} \right) |F|$, which contradicts our assumption. It follows that for a maximal pair of sequences $(c_i)_{i=1}^j$ and $(F_i)_{i=1}^j$ satisfying the relevant conditions, we have \begin{equation} \label{eq25} \left \vert I - W \right \vert \leq \frac{|J|}{k}. \end{equation} Fix such a maximal pair $(c_i)_{i=1}^j$ and $(F_i)_{i=1}^j$. Note that by our choice of $k$ in $(\ref{eq31})$ we have \begin{equation} \label{eq24} \frac{|J|}{k} \leq \frac{n}{k} \leq \frac{\eta n}{2}. \end{equation} Therefore if we put $P = G_\sigma - W$ then by $(\ref{eq15})$, $(\ref{eq25})$ and $(\ref{eq24})$ we have  \begin{align} |P| &\leq |G_\sigma - I| + \left \vert I - W \right \vert \nonumber \\ & \label{eq18} \leq \frac{\eta n}{2} + \frac{\eta n}{2} = \eta n. \end{align} 

\subsection{Controlling sofic entropy by naive entropy}

In this subsection we use the data previously constructed to bound the size of an appropriately separated subset of $\mathrm{Map}(\sigma,F,\delta)$ in terms of the separation numbers used to compute naive entropy. For $B \subseteq [n]$, let $d^\infty_B$ be the pseudometric on the collection of maps from $[n]$ to $X$ given by $d^\infty_B(\phi,\psi) = \max_{m \in B} d(\phi(m),\psi(m))$. Let $i \leq j$ and take an $\frac{\epsilon}{2}$-spanning set $V_i$ of $V$ of minimal cardinality with respect to the pseudometric $d^\infty_{B_i}$. We claim \[|V_i| \leq \exp\left(\frac{\kappa |F_i|}{4} \right).\] To see this, let $U$ be a maximal $\frac{\epsilon}{2}$-separated subset of $V$ with respect to $d^\infty_{B_i}$. Then $U$ is also $\frac{\epsilon}{2}$-spanning with respect to $d^\infty_{B_i}$ and hence $|V_i| \leq |U|$. For any two elements $\phi$ and $\psi$ of $V$ we have $c_i \in J \subseteq \Theta = \Theta_\psi = \Theta_\phi$. Since $F_i \subseteq F$ it follows from $(\ref{eq26})$ that $d\left(  \gamma^a \phi(c_i) , \phi \left( \gamma^\sigma c_i \right) \right) \leq \frac{\epsilon}{8}$ for all $\gamma \in F_i$, and similarly for $\psi$. So for all $\gamma \in F_i$ we have \begin{align} d\left(\gamma^a \phi(c_i),\gamma^a \psi(c_i)\right) &\geq d\left(\phi\left(\gamma^\sigma c_i\right) , \psi\left(\gamma^\sigma c_i\right)\right) -d\left(\gamma^a \phi(c_i),\phi\left(\gamma^\sigma c_i\right)\right) - d\left(\gamma^a \psi(c_i) , \psi\left(\gamma^\sigma c_i \right) \right) \nonumber \\ & \label{eq17} \geq d \left( \phi\left( \gamma^\sigma c_i \right), \psi\left( \gamma^\sigma c_i \right) \right) - \frac{\epsilon}{4}. \end{align}

Now, since $U$ is $\frac{\epsilon}{2}$-separated with respect to $d^\infty_{B_i}$, for any $\phi, \psi \in U$ we have \begin{equation} \label{eq27} d^\infty_{B_i}(\phi,\psi) = \max_{b \in B_i} d(\phi(b),\psi(b)) = \max_{\gamma \in F_i} d\left(\phi\left(\gamma^\sigma c_i\right),\psi \left(\gamma^\sigma c_i \right) \right) \geq \frac{\epsilon}{2}. \end{equation} By $(\ref{eq17})$ and $(\ref{eq27})$, \begin{align*} d_{F_i}(\phi(c_i),\psi(c_i)) &= \max_{\gamma \in F_i} d\left( \gamma^a \phi(c_i), \gamma^a \psi(c_i) \right) \\ & \geq \max_{\gamma \in F_i} \left( d \left( \phi \left(\gamma^\sigma c_i \right), \psi \left( \gamma^\sigma c_i \right) \right) - \frac{\epsilon}{4} \right) \\ &= \left( \max_{\gamma \in F_i}  d \left( \phi \left(\gamma^\sigma c_i \right), \psi \left( \gamma^\sigma c_i \right) \right) \right)- \frac{\epsilon}{4} \\ & \geq \frac{\epsilon}{2} - \frac{\epsilon}{4} = \frac{\epsilon}{4} .\end{align*}

It follows that $\{\phi(c_i): \phi \in U \}$ is an $\frac{\epsilon}{4}$-separated subset of $X$ with respect to $d_{F_i}$ of size $|U|$ and hence by Lemma \ref{lem4} we have \[|U| \leq \sep\left(X,\frac{\epsilon}{4},d_{F_i} \right) \leq \exp\left(\frac{\kappa |F_i|}{4} \right), \] and consequently \begin{equation} \label{eq20} |V_i| \leq \exp \left( \frac{\kappa |F_i|}{4} \right). \end{equation}

Now, take an $\frac{\epsilon}{2}$-spanning subset $V_P$ of $V$ of minimal cardinality with respect to $d^\infty_P$. Since a maximal $\frac{\epsilon}{2}$-separated subset is also $\frac{\epsilon}{2}$-spanning, we have \[ |V_P| \leq \sep\left( V, \frac{\epsilon}{2},d^\infty_P \right). \] For $i \in [n]$ write $\pi_i:X^{[n]} \to X$ for projection onto the $i$-coordinate. Note that if $S \subseteq V$ is $\frac{\epsilon}{2}$-separated with respect to $d^\infty_P$ then $\pi_i(S)$ is $\frac{\epsilon}{2}$-separated with respect to $d$ for each $i \in P$. Therefore $S$ is contained in a product of $|P|$ many $\frac{\epsilon}{2}$-separated subsets of $X$ and so \[ \sep\left( V, \frac{\epsilon}{2},d^\infty_P \right) \leq \sep \left(X,\frac{\epsilon}{2},d \right)^{|P|}  \] and thus $(\ref{eq18})$ implies \[|V_P| \leq \sep \left(X,\frac{\epsilon}{2},d \right)^{\eta n} \] and hence \begin{equation} \label{eq19} |V_P| \leq \exp\left( \frac{\kappa n}{4} \right) \end{equation} by our choice of $\eta$ in $(\ref{eq30})$.

\subsection{Conclusion}

Let $Z$ be the set of all maps $\phi:[n] \to X$ such that $\phi \rest P = \psi \rest P$ for some $\psi \in V_P$ and for each $i \leq j$ we have $\phi \rest B_i = \psi_i \rest B_i$ for some $\psi_i \in V_i$. Note that since we chose the sets $B_i = F^\sigma_i c_i$ to be pairwise disjoint, and the maps $\gamma \mapsto \gamma^\sigma c_i$ for $\gamma \in F_i$ are bijective, we have $\sum_{i=1}^j |F_i| \leq n$. Thus by $(\ref{eq20})$ and $(\ref{eq19})$ we have \begin{align} |Z| & \leq |V_P| \left( \prod_{i=1}^j |V_i| \right) \nonumber \\ & \leq \exp\left(\frac{\kappa n}{4} \right ) \left( \prod_{i=1}^j \exp \left(\frac{\kappa |F_i|}{4} \right)\right) \nonumber  \\ & = \exp\left(\frac{\kappa n}{4} + \frac{\kappa}{4} \left(\sum_{i=1}^j |F_i| \right) \right) \nonumber \\ & \leq \exp\left(\frac{\kappa n}{2} \right). \label{eq2} \end{align} Note that if $\phi \in V$, then by the hypothesis that $V_i$ is $\frac{\epsilon}{2}$-spanning for $V$ with respect to the metric $d^\infty_{B_i}$ we have that $\max_{b \in B_i} d(\phi(b),\psi_i(b)) \leq \frac{\epsilon}{2}$ for some element $\psi_i$ of $V_i$, and similarly for $P$ and $V_P$. Hence every element of $V$ is within $d^\infty$ distance $\frac{\epsilon}{2}$ of some element of $Z$. Define a map $f: V \to Z$ by letting $f(\phi)$ be any element of $Z$ within $d^\infty$ distance $\frac{\epsilon}{2}$ of $\phi$. Since $V$ is a subset of $D$ and we assumed that $D$ was $\epsilon$-separated with respect to $d^\infty$, it follows that $f$ is injective. Therefore we have $|V| \leq |Z|$. Then it follows from $(\ref{eq21})$ and $(\ref{eq2})$ that if $n \geq N$ then \begin{align*} \sep(\mathrm{Map}(F,\delta,\sigma_n),\epsilon,d^\infty) &= |D| \\& \leq \exp\left(\frac{\kappa n}{2} \right) |V| \\&  \leq \exp\left(\frac{\kappa n}{2} \right) |Z| \\ & \leq \exp\left(\frac{\kappa n}{2} \right) \exp\left( \frac{\kappa n}{2} \right) \\& = \exp\left(\kappa n \right).\end{align*} This concludes the proof of Theorem \ref{thm1}.

\bibliographystyle{plain}
\bibliography{bibliography}

Department of Mathematics\\
California Institute of Technology\\
Pasadena CA, 91125\\
\texttt{pjburton@caltech.edu}

\end{document}